\documentclass[11pt]{article}

\usepackage{amsfonts,amscd,amssymb,amsmath}
\usepackage{verbatim}
\newtheorem{theorem}{Theorem}
\newtheorem{proof}{Proof}

\date{}

\begin{document}

\title{$p$-Adic Invariant Summation \\ of Some $p$-Adic Functional Series}

\author{{Branko Dragovich and Nata\v sa \v Z. Mi\v si\'c*}
\\ {Institute of Physics, University of Belgrade, Pregrevica 118,  Belgrade, Serbia }\\
    {*Lola Institute, Kneza Vi\v seslava 70a, Belgrade, Serbia} }

\maketitle
\begin{abstract}
We consider summation of some finite and infinite  functional  $p$-adic series
with factorials. In particular, we are interested in the infinite series which are convergent for all primes $p$,
and have the same integer value for an integer argument.  In this paper, we present rather large class
of such $p$-adic  summable functional series with integer coefficients which contain factorials.
\end{abstract}

\section{Introduction}

It is well known that the series play important role in mathematics, physics and applications. When numerical
ingredients of the infinite series are rational numbers then they can be treated in any  $p$-adic as well as in real
number field, because rational numbers are endowed by real and $p$-adic norms simultaneously. In particular,
a real divergent series  deserves $p$-adic investigation when its $p$-adic sum is a rational for a rational argument.

Terms of many series in string theory, quantum field theory and quantum mechanics contain factorials. Such series are usually divergent in
the real case but convergent in $p$-adic ones. Motivated by this reason, different $p$-adic aspects of the series with factorials
are considered in \cite{bd1,bd2,bd3,bd4,bd5,bd6,bd7,bd8,bd9,bd10,bd11} and many summations are performed  in rational points. Also, notions of $p$-adic number
field invariant summation, rational summation, and adelic summation are introduced.

Note that $p$-adic numbers and $p$-adic analysis have been successfully applied in modern mathematical physics (from strings to complex systems
and the universe as a whole) and in some related fields (in particular in bioinformation systems, see, e.g. \cite{bd13}), see \cite{freund,vvz} for an early review and \cite{bd12} for a recent one.

In this paper we are interested in $p$-adic invariant summation of a class of infinite functional series which terms contain $n!$, i.e.
$\sum n! P_k (n; x) x^n ,$ where $P_k (n; x)$ are polynomials in $x$ of degree $k ,$ and which coefficients depend on $n .$ We  show that there exist
polynomials $P_k (n; x)$ for any degree $k ,$ such that for any $x \in \mathbb{Z}$ the corresponding sums are also  rational numbers. Moreover, we have found
recurrent relations how to calculate all ingredients of such  $P_k (n; x)$. The obtained summation formula is a generalization of previously
derived one when $x = 1$ (see \cite{bd10}) in $\sum n! P_k (n; x) x^n .$  Many results are illustrated by some simple examples.

 All necessary general information on $p$-adic series can be found in standard books on $p$-adic analysis, see, e.g. \cite{schikhof}.


\section{$p$-Adic Functional Series with Factorials}

We consider $p$-adic functional series of the form \begin{align}
S_k (x) = \sum_{n = 0}^{+\infty} n! \, P_k (n; x)\, x^n = P_k (0;x) + 1!\, P_k (1;x)\, x  + 2!\, P_k (2;x)\, x^2 + \dots \,, \label{2.1}
\end{align}
where
\begin{align} \label{2.2}
& P_k (n;x) =C_k (n)\, x^k + C_{k-1} (n)\, x^{k-1} + \dots + C_1(n)\, x + C_0(n)\,,  \\
& n! = 1\cdot 2 \cdots n, \, (0!=1), \quad k \in \mathbb{N}_0 = \mathbb{N}\cup\{ 0 \}, \quad x \in \mathbb{Q}_p \nonumber
\end{align}
and $C_j (n), \, 0 \leq j \leq k ,$ are some polynomials in $n$ with integer
coefficients. Mainly we are interested for which polynomials $P_k(n; x)$ we have that if $x \in \mathbb{Z}$ then the sum of the series
(1) is $S_k (x) \in \mathbb{Z}$, i.e. $S_k (x)$ is also an integer the same in all $p$-adic cases. Since polynomials $P_k(n; x)$ are determined
by polynomials  $C_j (n), \, 0 \leq j \leq k ,$ it means that one has to find these $C_j (n), \, 0 \leq j \leq k .$  Our task is to find connection
between polynomial $P_k(n; x)$ and sum $S_k (x),$ which are also a polynomial.

\subsection{Convergence of the $p$-Adic Series}

Note that necessary and sufficient condition for $p$-adic power series
to be convergent \cite{schikhof,vvz} is
\[
f(x) = \sum_{n=1}^{+\infty} a_n x^n ,   \quad a_n \in \mathbb{Q} \subset \mathbb{Q}_p ,
\quad x \in \mathbb{Q}_p ,  \quad |a_n x^n|_p \to 0 \,\, as \,\, n
\to \infty .
\]
Necessary condition follows like in the real case, and sufficient condition is a consequence
of ultrametricity. Recall that  ultrametric distance satisfies the strong triangle inequality
$
d(x,y) \leq max \{d(x,y), d(y,z)\}.
$   $p$-Adic distance $d_p (x,y) = |x-y|_p$ is ultrametric one, where $|\cdot|_p$ is $p$-adic norm.
According to the Cauchy criterion, the sufficiency follows from $|a_n + a_{n+1} + \cdots + a_{n+m}|_p \leq max_{n\leq i\leq n+m} |a_i|_p = |a_n|_p \to 0$ when
$n \to \infty .$

 The functional series (1) contains $n!$, hence to investigate its convergence one has to know $p$-adic norm of $n! .$
 First, let us find a power $M(n)$ by which  prime $p$ is contained in $n!$  (see, e.g. \cite{vvz}). Let $n = n_0 + n_1 p +
... + n_r p^r $ and $ s_n = n_0 + n_1 + ... + n_r $ denotes the sum of digits in expansion of a natural number $n$ in base $p$. We also denote by $[x]$
the integer part of a number $x$.   Then
\begin{align} M(n) &= \Big[\frac{n}{p}\Big] + \Big[\frac{n}{p^2}\Big] +
\cdots + \Big[\frac{n}{p^r}\Big] \nonumber \\
&= \frac{n-n_0}{p} + \frac{n-n_0 -n_1 p}{p^2} + \cdots +
\frac{n-n_0 -n_1 p  - \cdots - n_{r -1} p^{r-1}}{p^r} \nonumber \\
& = \frac{n}{p} \Big(1 + \frac{1}{p} + \cdots + \frac{1}{p^{r-1}}\Big) -
\frac{n_0}{p} \Big(1 + \frac{1}{p} + \cdots + \frac{1}{p^{r-1}}\Big) \nonumber \\
& - \frac{n_1}{p} \Big(1 + \frac{1}{p} + \cdots + \frac{1}{p^{r-2}}\Big) -
\frac{n_2}{p} \Big(1 + \frac{1}{p} + \cdots + \frac{1}{p^{r-3}}\Big)\nonumber \\
& - \cdots - \frac{n_{r-1}}{p} .  \label{2.3}
 \end{align}

 Performing summation in \eqref{2.3},  it follows

\begin{align} M(n) &= n \frac{1-p^{-r}}{p-1} - n_0 \frac{1-p^{-r}}{p-1}
- n_1 \frac{1-p^{-r+1}}{p-1} - n_2 \frac{1-p^{-r+2}}{p-1}
\nonumber \\ &- \cdots - n_{r-1} \frac{1-p^{-1}}{p-1} -
n_r\frac{1}{p-1} + n_r\frac{1}{p-1} \nonumber \\ &= \frac{n -
s_n}{p-1} \label{2.4}
\end{align}

 Hence, one obtains
 \begin{align} &n! = m\, p^{M(n)} = m\,  p^{\frac{n -s_n}{p-1}} ,\quad p \nmid  m , \nonumber  \\
 &|n!|_p = p^{-\frac{n -s_n}{p-1}} . \label{2.5} \end{align}

Consider $p$-adic norm of the general term in \eqref{2.1} with $P_k (n; 1)$, i.e.

\begin{align}  |n! P_k(n; 1)  x^n|_p \leq  |n! x^n|_p = p^{- \frac{n - s_n}{p-1}} |x|_p^n =
\Big( p^{- \frac{n - s_n}{n(p-1)}} |x|_p  \Big)^n \nonumber \end{align}
Then, because convergence requires $|n! P_k(n; 1)  x^n|_p \to 0$ as $n \to 0,$ one has the following relations:
\begin{align}
p^{- \frac{n - s_n}{n(p-1)}} |x|_p < 1, \quad  |x|_p < p^{
\frac{n - s_n}{n(p-1)}} \to p^{\frac{1}{p-1}} \,\, as \,\, n \to
\infty . \nonumber
\end{align}

 Hence, the region of convergence $D_p$ of the power series $\sum P_k(n; 1)n!x^n$
is $D_p =\mathbb{Z}_p$, i.e. $|x|_p \leq 1$, because norms in
$\mathbb{Q}_p$ are of the form $p^\nu , \,\, \nu \in \mathbb{Z}$
and $1 < p^{\frac{1}{p-1}}. $ The region of convergence $|x|_p \leq 1$ is valid also
for the series \eqref{2.1}, because $|P_k(n; x)|_p \leq |P_k(n; 1)|_p$ when $|x|_p \leq 1.$

 Since $\bigcap_p \mathbb{Z}_p = \mathbb{Z}$, it means that the
infinite series $\sum P_k(n; x) n! x^n$ is simultaneously convergent for all
integers and all $p$-adic norms.

\section{Summation at Integer Points}

We are interested now in determination of the polynomials $P_k(n; x)$ and the corresponding sums $S_k (x) = Q_k (x)$ of the infinite series
\eqref{2.1}, where
\begin{align}
Q_k (x)= q_k \, x^k \, + \, q_{k-1} \, x^{k-1} \, + \cdots + \, q_1 \, x^1 \, + \, q_0   \label{3.1}
\end{align}
are also some polynomials related to $P_k(n; x) ,$  so that $P_k(n; x) $ and $Q_k (x)$ do not depend on concrete
$p$-adic consideration and that they are valid for all $x \in \mathbb{Z} .$

The simplest illustrative case \cite{schikhof} of $p$-adic invariant summation of the infinite series is

\begin{align} \sum_{n \geq 0} n! \, n = 1! 1 + 2! 2 + 3! 3 + ... = - 1 \label{3.2} \end{align}
and obtains from \eqref{2.1} taking  $x = 1, \, \, P_1(n; 1) = n,$ which gives $Q_1 (1) = - 1.$
 To prove  \eqref{3.2}, one can use any one of the following two properties:
\[  n! n = (n+1)! - n!  \,, \quad \, \qquad \sum_{n = 1}^{N-1} n! \, n  = -1 + N! , \]
where $n! \to 0$ as $n \to \infty .$

In the sequel we shall develop and apply the corresponding method of the definite partial sums.

\subsection{The Definite Partial Sums}

\begin{theorem}
Let
\begin{equation}
A_k (n; x) = a_k(n) x^k + a_{k-1}(n) x^{k-1} + \cdots + a_1(n) x + a_0(n)  \label{3.3}
\end{equation}
be a polynomial with coefficients $a_j (n) ,\,\, 0\leq j \leq k,$ which are polynomials in $n$ with integer coefficients.
Then there exists such polynomial $A_{k-1} (N; x), \,\, N \in \mathbb{N},$ that equality
\begin{align} \sum_{n=0}^{N-1} n! \, [ n^k x^k + U_k (x) ] x^n = V_k (x) + N! x^N A_{k-1}(N; x)   \label{3.4}  \end{align}
holds, where
\begin{align} U_k(x) = x A_{k-1} (1; x) - A_{k-1} (0; x) \,, \quad V_k(x) = - A_{k-1} (0; x).  \label{3.5}\end{align}
\end{theorem}

\begin{proof}
Let us consider summation of the following finite power series:
\[ S_k(N; x) = \sum_{n=0}^{N-1} n! n^k  x^n  = \delta_{0k} + \sum_{n=1}^{N-1} n! n^k  x^n, \quad   k \in \mathbb{N}_0  .  \]

We derive the following recurrent formula:
\begin{align}
S_k(N; x) &= \delta_{0k} + \sum_{n=0}^{N-2} (n+1)! (n+1)^k x^{n+1} = \delta_{0k} + x +
\sum_{n=1}^{N-1} n! (n+1)^{k+1} x^{n+1} - N! N^k x^N \nonumber \\
&= \delta_{0k} + x + x \sum_{n=1}^{N-1} n! x^n \sum_{\ell =0}^{k+1} \Big(\begin{aligned} k&+1 \\
&\ell \end{aligned}\Big) n^\ell - N! N^k x^N \nonumber \\ &= \delta_{0k} +
x S_0(N;x) + x  \sum_{\ell =1}^{k+1} \Big(\begin{aligned} k&+1\\
& \ell \end{aligned}\Big) S_\ell (N; x) - N! N^k x^N  .  \label{3.6}
\end{align}

The above recurrent formula gives possibility to calculate sum
$S_{k+1} (N; x)$ knowing all preceding sums $S_{\ell} (N; x), \,
\, \ell = 1, 2, ..., k$ as function of $S_{0} (N; x)$. These sums
have the form
\[ \sum_{n=0}^{N-1} n! \, [ n^k x^k + U_k (x) ] x^n = V_k (x) + N! x^N A_{k-1}(N; x) , \, \, k \in \mathbb{N} ,  \]
which can be rewritten  as
\begin{align}  S_k (N; x) = - U_k(x) S_0 (N; x) x^{-k} + V_k (x) x^{-k} + N! A_{k-1} (N; x) x^{N-k} , \, \, k \in \mathbb{N} , \label{3.7} \end{align}
where $U_k (x)$ and  $V_k (x)$ are some polynomials in x of the
degree  $k$, and $A_k (N; x)$ is a polynomial in $x$ which
coefficients of $x^n$ are polynomials in $N$ of degree $n$.

Substituting the above expressions of the sums \eqref{3.7} into the recurrence formula  \eqref{3.6}, we
obtain recurrence formulas for $U_{k+1} (x), V_{k+1} (x)$ and
$A_{k}(N; x):$
\begin{equation}
\sum_{\ell =1}^{k+1} \Big(\begin{aligned}k&+1\\ &\ell\end{aligned}\Big) x^{k-\ell+1} \, U_\ell(x) - U_k(x) - x^{k+1} = 0 , \label{3.8}
\end{equation}
\begin{equation}
\sum_{\ell =1}^{k+1} \Big(\begin{aligned}k&+1\\ &\ell\end{aligned}\Big) x^{k-\ell+1} \, V_\ell(x) - V_k(x)  = 0 , \label{3.9}
\end{equation}
\begin{equation}
\sum_{\ell =1}^{k+1} \Big(\begin{aligned} k&+1\\ &\ell\end{aligned}\Big) x^{k-\ell+1} \, A_{\ell-1}(N;x) - A_{k-1}(N;x) - N^kx^{k} = 0 . \label{3.10}
\end{equation}

Putting $N = 0$ in \eqref{3.10}, it follows that $V_k (x)$ is proportional to $A_{k-1} (0; x) .$  After explicit  calculation of  $A_0 (0; x)$ and $V_1 (x) ,$ we conclude that
$V_k (x) = -  A_{k-1} (0; x)$ for all $k \in \mathbb{N} .$  On the other hand, equality $U_k(x) = x A_{k-1} (1; x) - A_{k-1} (0; x)$ in \eqref{3.5} obtains if we replace $N =1$ in \eqref{3.10} and multiply it by $x$, and then subtracting  from such expression
recurrence relation \eqref{3.10} with $N= 0 .$

\end{proof}

It is worth emphasizing that equality \eqref{3.4} is valid in real and all $p$-adic cases. The central role in this equality plays
polynomial $A_{k} (N; x) $ which is solution of the recurrence relation \eqref{3.10}.
 When $N \to \infty$ in \eqref{3.4}, the term with polynomial
$A_{k-1} (N; x)$ $p$-adically disappears giving the sum of the following $p$-adic infinite functional series:
\begin{equation}
 \sum_{n=0}^{\infty} n! \, [ n^k x^k + U_k (x) ] x^n = V_k (x) .    \label{3.11}
\end{equation}
This equality has the same form for any $k \in \mathbb{N} ,$ and polynomials $U_k (x)$ and $V_k (x)$ separately have the same values in all $p$-adic cases for any
$x \in \mathbb{Z} .$ In other words, nothing depends on $p$-adic properties in \eqref{3.11} if $x \in \mathbb{Z}$, i.e. this is $p$-adic invariant
result. This result gives us the possibility to present a general solution of the problem posed on $p$-adic invariant summation of the series \eqref{2.1}.

\begin{theorem}
The functional series \eqref{2.1}  has $p$-adic invariant sum
\begin{equation}
\sum_{n=0}^{\infty} n! \, P_k (n; x) x^n = Q_k (x)   \label{3.12}
\end{equation}
if
\begin{equation}
P_k (n; x) = \sum_{j =1}^k C_j [n^j x^j + U_j (x)]    \quad \text{and} \quad   Q_k (x) = \sum_{j =1}^k C_j \, U_j (x),  \label{3.13}
\end{equation}
where $C_j, \, x \in \mathbb{Z} .$
\end{theorem}

Recurrent formulas \eqref{3.8}--\eqref{3.10} enable to calculate these polynomials for all $k \in \mathbb{N}$,
knowing initial expressions of  $U_1 (x),  V_1 (x)$ and $A_0 (N; x)$, which can be
obtained from  \eqref{3.6} and they are: $U_1(x) = x-1, \, V_1(x) =
-1$ and $A_0 (N; x) = 1$.

For the first five values of degree $k$ we have obtained the following explicit expressions.

\begin{itemize}
\item  $k = 1$
\begin{align}
&U_1(x) = x-1, \nonumber\\
&V_1(x) = -1, \nonumber \\
&A_0 (n; x) = 1.  \label{3.14}
\end{align}
\item  $k = 2$
\begin{align}
&U_2(x) = - x^2+ 3x -1, \nonumber\\
&V_2(x) =  2x - 1, \nonumber\\
&A_1(n; x) = (n -2) x  + 1 .   \label{3.15}
\end{align}
\item $k = 3$
\begin{align}
&U_3(x) =  x^3 - 7x^2  + 6x -1, \nonumber\\
&V_3(x) = -3 x^2 +5 x -1, \nonumber\\
&A_2(n; x) = (n^2 -3n +3)x^2 + (n -5) x  + 1 .  \label{3.16}
\end{align}
\item   $k = 4$
\begin{align}
&U_4(x) = -x^4 + 15x^3 - 25 x^2+ 10x -1, \nonumber\\
&V_4(x) =  4x^3 -17x^2 + 9x -1, \nonumber\\
&A_3(n; x) = (n^3 - 4n^2 +6n -4)x^3 +(n^2 -7n +17)x^2 +(n -9) x  + 1 .  \label{3.17}
\end{align}
\item    $k = 5$
\begin{align}
&U_5(x) = x^5 -31x^4 + 90 x^3 - 65 x^2+ 15x -1,  \nonumber \\
&V_5(x) = -5 x^4 + 49x^3 - 52 x^2 + 14x -1,   \nonumber\\
&A_4(n; x) =  (n^4 -5n^3 +10n^2 -10n +5)x^4 +  (n^3 - 9n^2 +31
n-49)x^3 \nonumber \\ & \qquad  \qquad + (n^2 -12n +52)x^2 + (n -14) x + 1 .  \label{3.18}
\end{align}
\item  $k = 6$
\begin{align}
&U_6(x) = -x^6 + 63x^5 -301 x^4 +350 x^3 - 140 x^2 +21 x -1, \nonumber \\
&V_6(x) = 6 x^5 -129x^4+246x^3 - 121 x^2 + 20 x -1,   \nonumber  \\
&A_5(n; x) = (n^5 -6 n^4 +15 n^3 - 20 n^2 +15 n -6) x^5 + (n^4 - 11n^3  \nonumber \\ & \qquad \qquad
+ 49n^2 - 111n  +129)x^4 + (n^3 - 15n^2 + 88 n - 246) x^3 \nonumber \\ & \qquad \qquad + (n^2 - 18 n +121) x^2 + (n - 20) x + 1 .   \label{3.18a}
\end{align}
\end{itemize}

It is already noted that polynomial $A_k (n; x)$ plays central role in $p$-adic invariant summation of the series \eqref{2.1}.
$A_k (n; x)$ as well as $U_k (x)$ and $V_k (x)$ can be written it the compact form
\begin{equation}
A_k (n; x) = \sum_{\ell = 0}^k A_{k\ell} (n) \, x^\ell , \quad U_k (x)= \sum_{\ell = 0}^k  U_{k\ell} \, x^\ell ,   \quad V_k (x) = \sum_{\ell = 0}^k  V_{k\ell} \, x^\ell ,      \label{3.18b}
\end{equation}
where $A_{k\ell} (n)$ is polynomial in $n$ of degree $\ell$ with $n^\ell$ as the term of highest degree.

Then the following properties hold:
\begin{itemize}
\item  $A_k(n; 0) = A_{k0} = 1 ,  \, \, k = 0, 1, 2, ...$
\item  $A_{kk}(1) = (-1)^k , \, \,  k = 0, 1, 2, ...$
\item  $ U_{k0} = V_{k0} =  -1 , \, \, k =  1, 2, ...$
\item  $U_{kk}= (-1)^{k+1} , \, \,  k =  1, 2, ... $
\item  $V_{kk} = (-1)^{k} k , \, \,  k =  1, 2, ...$
\end{itemize}

\subsection{Connection with Bernoulli Numbers and Some \\ Simple Examples}

Using the Volkenborn integral one can make connection of our summation formulas with the Bernulli numbers. By definition the Volkenborn integral \cite{schikhof} is
$$
\int_{\mathbb{Z}_p} f(x)\, dx = \lim_{x \to \infty} p^{-n} \sum_{j=0}^{p^n -1} f(j) ,
$$
where $f$ is a $p$-adic continuous function. One of its properties is
\begin{equation}
\int_{\mathbb{Z}_p} x^n \, dx  = B_n ,  \quad n = 0, 1, 2, ... , \label{3.19a}
\end{equation}
where $B_n$ are the Bernoulli numbers which can be defined by the recurrent relation
$$
\sum_{j=0}^{n-1}\Big(\begin{aligned} &n \\ &j \end{aligned} \Big) B_j =0 , \, \, B_0 = 1 . \nonumber
$$
Note that $B_1 = -\frac{1}{2}, \quad B_3 = B_5 = ... = B_{2n+1} = 0 .$

As an illustration of the above summation formula \eqref{3.11}, we present three simple (k = 1, 2, 3) examples
including connection with the Bernoulli numbers.

\begin{itemize}

\item  $ k = 1$
\begin{align}
 &\sum_{n=0}^{\infty} n! \, [(n + 1)x  - 1]\, x^n = -1 ,  \quad x \in \mathbb{Z} , \\   \label{3.19}
&\sum_{n=0}^{\infty} n! \, n \,  x^n = -1 ,  \quad x = 1, \nonumber  \\
&\sum_{n=0}^{\infty} n! \, (-1)^n \,(n + 2)  =  1 ,  \quad x = - 1, \nonumber \\
&\sum_{n=0}^{\infty} n! \, [(n + 1) B_{n+1}  - B_n] = -1 . \nonumber
\end{align}

\item $k = 2$
\begin{align}
 &\sum_{n=0}^{\infty} n! \, [(n^2 - 1) x^2 + 3x - 1]\, x^n = 2x -1 ,  \quad x \in \mathbb{Z} , \\   \label{3.20}
&\sum_{n=0}^{\infty} n! \, (n^2 + 1) = 1 ,  \quad x = 1 ,     \nonumber \\
&\sum_{n=0}^{\infty} (-1)^n \, n! \, (n^2 - 5) = -3 ,  \quad x = -1 , \nonumber \\
&\sum_{n=0}^{\infty} n! \, [(n^2 - 1) B_{n+2} + 3 B_{n+1} - B_n] = -2 . \nonumber
\end{align}

\item $k = 3$
\begin{align}
 &\sum_{n=0}^{\infty} n! \, [(n^3 + 1) x^3  - 7 x^2 + 6x - 1]\, x^n = -3x^2 + 5x -1 ,  \quad x \in \mathbb{Z} , \\   \label{3.21}
&\sum_{n=0}^{\infty} n! \, (n^3 - 1) x^3  = 1 ,  \quad x = 1 ,   \\
&\sum_{n=0}^{\infty} (-1)^n \, n! \, (n^3 + 15) =  9 ,  \quad x = -1 ,    \\
&\sum_{n=0}^{\infty} n! \, [(n^3 + 1) B_{n+3}  - 7 B_{n+2} + 6 B_{n+1} - B_n] = -4 .
\end{align}
\end{itemize}

The above series with the Bernoulli numbers are $p$-adic convergent, because $|B_n|_p \leq p$ (see \cite{schikhof}, p. 172).

\section{Concluding Remarks}
\bigskip

The main results presented in this paper are Theorem 1 and Theorem 2. These results are  generalizations of an earlier
result for $x = 1$ \cite{bd10} (see also \cite{bd9,bd11} for some partial generalizations).

Finite series with their sums  \eqref{3.4} are valid for real and $p$-adic numbers. When $n \to \infty$ the corresponding infinite
series are divergent in real case, but are convergent and have  the same sums in all $p$-adic cases. This fact can be used to extend
these sums to the real case. Namely, the sum of a divergent series depends on the way of its summation and here it can be used rational sum
of the series valuable in all $p$-adic number fields. This way of summation of real divergent series was introduced for the first time in
\cite{bd2} and called adelic summability. How this adelic summability is important depends on its practical future use in some concrete examples.

The simplest infinite series with $n!$ is $\sum n! .$ It is convergent in all $\mathbb{Z}_p ,$ but has not $p$-adic invariant sum. Even it is not
known so far does it has a rational sum in any $\mathbb{Z}_p .$  Rationality of this series and $\sum n! n^k x^n$ was discussed in \cite{bd9}. The series
$\sum n!$ is also related to Kurepa hypothesis which  states $(!n, n!) = 2, \quad 2 \leq n \in \mathbb{N} ,$  where $!n = \sum_{j=0}^{n-1} j!$. Validity
of this hypothesis is still an open problem in number theory. There are many equivalent statements to the Kurepa hypothesis, see  \cite{bd10} and references
therein. From $p$-adic point of view, the Kurepa hypothesis reads: $\sum_{j=0}^{\infty} j!  = n_0 + n_1 p + n_2 p^2 + \cdots ,$ where  digit $n_0 \neq 0$ for all primes $p \neq 2 .$

When $x = 1$, then \eqref{3.11} becomes
\begin{equation}
 \sum_{n=0}^{\infty} n! \, [ n^k x^k + u_k ] x^n = v_k ,    \label{3.22}
\end{equation}
where $u_k = U_k (1)$ and $v_k = V_k (1)$ are some integers. Equality \eqref{3.22} was introduced in \cite{bd5}, and properties of $u_k$ and $v_k$ are investigated in series of papers by Dragovich (see references \cite{bd8,bd9,bd10,bd11}). In \cite{murty} relationships  of $u_k$ with the Stirling and the Bell numbers are established, and $p$-adic irrationality of $\sum_{n \geq 0} n! n^k$ was discussed (\cite{subedi1,subedi2,alexander}). Note that the following sequences are related to some real (combinatorial) problems \cite{sloane}:
\begin{align}
&A_{k-1} (0;1) = -V_k (1) = - v_k : \, \, 1, -1, -1, 5, -5, -21, ...        \quad  \text{see  \, A014619} \\
&A_{k-1} (0;-1) = -V_k (-1) = - \bar{v}_k : \, \, 1, 3, 9, 31, 121, 523, ...   \quad  \text{see  \, A040027} \\
&A_{k-1} (1;1) - A_{k-1}(0;1) = U_k (1) = u_k : \, \, 0, 1, -1, -2, 9, -9, ...  \quad  \text{see  \, A000587} \\
&A_{k-1} (1;-1) + A_{k-1}(0;-1) = -U_k (-1) = -\bar{u}_k : \, \, 2, 5, 15, 52, 203, ... \quad  \text{see  \, A000110}.
\end{align}

There are many possibilities how results obtained in this paper can be generalized in future research of $p$-adic invariant summation of
some real divergent series. Various  aspects of the sequence of polynomials $A_k (n; x)$ deserve to be analyzed.
It would be also interesting to investigate relations between the Bernoulli numbers using
the above finite sums of the form \eqref{3.4}.

\section*{Acknowledgments}
Work on this paper was supported by Ministry of Education, Science and Technological Development of the Republic of Serbia, projects: OI 174012, TR 32040 and
TR 35023.

\bibliographystyle{amsplain}

\end{document}